\documentclass[10pt]{amsart}
\usepackage{amssymb}
\usepackage{amscd}
\usepackage{mathrsfs}

\def\today{\number\day\space\ifcase\month\or   January\or February\or
   March\or April\or May\or June\or   July\or August\or September\or
   October\or November\or December\fi\   \number\year}

\theoremstyle{definition}
\newtheorem{thm}{Theorem}[section]
\newtheorem{lem}[thm]{Lemma}
\newtheorem{prp}[thm]{Proposition}
\newtheorem{dfn}[thm]{Definition}
\newtheorem{cor}[thm]{Corollary}

\newtheorem{rmk}[thm]{Remark}

\renewcommand{\qed}{\rule{0.4em}{2ex}}

\newcommand{\beq}{\begin{equation}}
\newcommand{\eeq}{\end{equation}}
\newcommand{\beqr}{\begin{eqnarray*}}
\newcommand{\eeqr}{\end{eqnarray*}}
\newcommand{\bal}{\begin{align*}}
\newcommand{\eal}{\end{align*}}
\newcommand{\bei}{\begin{itemize}}
\newcommand{\eei}{\end{itemize}}

\newcommand{\ep}{\varepsilon}

\newcommand{\ph}{\varphi}

\newcommand{\Z}{{\mathbb{Z}}}
\newcommand{\R}{{\mathbb{R}}}
\newcommand{\C}{{\mathbb{C}}}
\newcommand{\N}{{\mathbb{N}}}

\newcommand{\tsr}{{\mathrm{tsr}}}

\newcommand{\ca}{C*-algebra}

\newcommand{\Aut}{{\mathrm{Aut}}}

\newcommand{\Index}{{\mathrm{Index}}}

\newcommand{\dist}{{\mathrm{dist}}}

\title[The Jiang-Su absorption for inclusions of unital C*-algebras]{The Jiang-Su absorption for inclusions of unital C*-algebras}
\author{Hiroyuki Osaka$^*$}
\date{28 April, 2014}
\thanks{$^*$Research of the first author partially supported by the JSPS grant for Scientific Research No.23540256}
\address{ Department of Mathematical Sciences\\
  Ritsumeikan University\\ Kusatsu, Shiga, 525-8577  Japan}

\email[]{osaka@se.ritsumei.ac.jp}

\address{}

\author{Tamotsu Teruya}

\address{Faculty of Education, Gunma University, 4-2 Aramaki-machi,
Maebashi City, Gunma, 371-8510, Japan}

\email[]{teruya@gunma-u.ac.jp}


\begin{document}
\maketitle

\begin{abstract}
We introduce the tracial Rokhlin property for a conditional expectation for an inclusion 
of unital C*-algebras $P \subset A$ with index finite, and show that an action $\alpha$ 
from a finite group $G$ on a simple unital C*-algebra $A$ has the tracial Rokhlin property 
in the sense of N.~C.~Phillips 
if and only if the canonical conditional expectation $E\colon A \rightarrow A^G$ has the tracial
Rokhlin property. 
Let $\mathcal{C}$ be  a class of infinite dimensional stably finite separable unital C*-algebras
which is closed under the following conditions:
\begin{enumerate}
\item[$(1)$]
$A \in {\mathcal C}$ and $B \cong A$, then $B \in \mathcal{C}$.
\item[$(2)$]
If $A \in \mathcal{C}$ and $n \in \N$, then $M_n(A) \in \mathcal{C}$.
\item[$(3)$] 
If $A \in \mathcal{C}$ and $p \in A$ is a nonzero projection, 
then $pAp \in \mathcal{C}$.
\end{enumerate}
We prove that for any simple unital C*-algebra $A$ which is also local tracial $\mathcal{C}$-algebra in the sense of Fan and Fang, 
if any C*-algebra in $\mathcal{C}$ is weakly semiprojective and $E\colon A \rightarrow P$ is 
of finite index type with the tracial Rokhlin property, $P$ is a unital local tracially  $\mathcal{C}$-algebra. 
The main result is that if  $A$ is simple, separable, unital nuclear, the Jiang-Su absorbing 
and $E\colon A \rightarrow P$ has the trcial Rokhlin poperty, then $P$ is the Jiang-Su absorbing.
As an application, when an action $\alpha$ 
from a finite group $G$ on a simple unital C*-algebra $A$ has the tracial Rokhlin property, 
then for any subgroup $H$ of $G$ the fixed point algebra $A^H$ and the crossed product algebra 
$A \rtimes_{\alpha_{|H}} H$ has the Jiang-Su absorbing.
We also show that the strictly comparision property for a Cuntz semigroup $W(A)$ 
is hereditary to $W(P)$ if  $A$ is simple, separable, exact, unital, 
and $E\colon A \rightarrow P$ has the trcial Rokhlin property. 
\end{abstract}




\section{Introduction}
The purpose of this paper is to introduce the tracial Rokhlin property for an
inclusion of separable simple unital C*-algebras $P \subset A$ with finite index 
in the sense of \cite{Watatani:index}, and prove theorems of the following type. 
Suppose that $A$ belongs to a class of C*-algebras characterized by 
some structural property, such as tracial rank zero in the sense of \cite{Lin:tracial}. 
Then $P$ belongs to the same class. 
The classes we consider include:

\begin{itemize}
\item Simple C*-algebras with real rank zero or stable rank one.
\item Simple C*-algebras with tracial rank zero or tracial rank less than or equal to one.
\item Simple C*-algebras with the Jiang-Su algebra $\mathcal{Z}$ absorption.
\item Simple C*-algebras for which the order on projections is determined by traces.
\item Simple C*-algebras with the strict comparison property for the Cuntz semigroup.
\end{itemize}

The 3rd condition and 5th condition are important properties related to 
Toms and Winter's conjecture, that is, the properties of strict
comparison, finite nuclear dimension, and Z-absorption are equivalent for separable
simple infinite-dimensional nuclear unital C*-algebras (\cite{T} and \cite{WZ}).

We show that an action $\alpha$ from a finite group $G$ on a simple unital C*-algebra 
$A$ has the tracial Rokhlin property in the sense of \cite{Phillips:tracial} if and only if the 
canonical conditional expectation $E\colon A \rightarrow A^G$ has the tracial Rokhlin property 
for an inclusion $A^G \subset A$. When an action $\alpha$ from a finite group 
on a (not necessarily simple) unital C*-algebra has the Rokhlin property in the sense of \cite{Izumi:Rohlin1}, 
all of the above results are proved in \cite{OT}\cite{OT2}.

The essential observation is already done in the proof of Theorem 2.2 of \cite{Phillips:tracial} the crossed product 
$A \rtimes_\alpha G$ ($ = C^*(G, A, \alpha)$ in \cite{Phillips:tracial}) has a local approximation property by C*-algebras 
stably isomorphic to homomorphic images of $A$. 
Since the Jiang-Su algebra $\mathcal{Z}$ belongs to classes of direct limits of semiprojective building 
blocks in \cite{Jiang-Su:absorbing}, technical difficulties arises because we must treat arbitrary homomorphic images 
in the approximation property. (Homomorphic images of semiprojective C*-algebras need not be semiprojective.)
In \cite{YF} they introduced the unital local tracial 
$\mathcal{C}$ property, which is generalized one of a local $\mathcal{C}$ property in \cite{OP:Rohlin}, 
for proving  that a C*-algebra with a local approximation property by 
homomorphic images of a suitable class $\mathcal{C}$ of semiprojective C*-algebras can be written as a direct limit of algebras in the class. 
When each homomorphism is injective, the unital local tracial $\mathcal{C}$ property is equivalent to the 
tracial approximation property in \cite{EN}. 
Note that when an action $\alpha$ from a finite group $G$ on a simple unital 
C*-algebra $A$, the tracial approximation property for $A$ is inherited to the crossed product algebra $A \rtimes_\alpha G$.
(See \cite{YF} and Theorem~3.3.)

We know of several results like those above  for tracial approximation in the literature:
for stable rank one (Theorem~4.3 of \cite{EN} and \cite{FF}), for real rank zero (\cite{FF}),
for the $\mathcal{Z}$-absorption (Corollary~5.7 of \cite{HO}), 
for the Blackadar comparson property (Theorem~4.12 in \cite{EN}).

The paper is organized as follows. In section 2 we introduce the notion of a unital local tracial $\mathcal{C}$-algebra and 
a tracial approximation class (TA$\mathcal{C}$ class) and we show in section 3 that when an action $\alpha$ 
from a finite group $G$ on a simple unital C*-algebra $A$ has the tracial Rokhlin property, then the crossed product algebra 
$A \rtimes_\alpha G$ belongs to the class TA$\mathcal{C}$ for $A$ in TA$\mathcal{C}$. 
In section 4 we introduce the tracial Rokhlin property for an inclusion $P \subset A$ of unital C*-algebras and 
show that if $A$ is a simple local tracial $\mathcal{C}$-algebra, then so is $P$ (Theorem~4.11).
In particular, if $A$ has tracial topological rank zero (resp. less than or equal to one), so does $P$ (Corollary~4.12).
In section 5 we preset  a main theorem, that is, let an inclusion $P \subset A$ of 
separable simple nuclear unital C*-algebras be of index finite type  with the tracial Rokhlin property, 
and suppose that $A$ is $\mathcal{Z}$-absorbing, then so does $P$ (Theorem~5.4). 
As an application, any fixed point algebra $A^H$ for any subgroup $H$ of a finite group $G$ is $\mathcal{Z}$-absorbing under the 
assumption that an action $\alpha$ from $G$ on a simple nuclear unital C*-algebra $A$ 
such that $A$ is $\mathcal{Z}$-absorbing (Corollary~5.5). 
Before treating the strict comparison for a Cuntz semigroup, 
we consider the Cuntz equivalent for positive elements and show that under the assumption that  an inclusion $P \subset A$ of 
unital C*-algebras has the tracial Rokhlin property, if for $n \in \N$ and positive elements $a, b \in M_n(P)$ 
$a$ is subequivalent to $b$ in $M_n(A)$, then $a$ is subequivalent to $b$ in $M_n(P)$ (Proposition~6.2).
Finally, we consider the strict comparison property for a Cuntz 
semigroup and show that the strict comparison property is inherited to $P$ when an inclusion $P \subset A$ of 
simple separable exact unital C*-algebras has the tracial Rokhlin property and $A$ has the strict comparison (Theore~7.2). 
Using the similar argument we show that if an inclusion $P \subset A$ of separable simple unital C*-algebras 
has the tracial Rokhlin property and  the order on projections in $A$ is determined by traces, then 
the order on projections in $P$ is determined by traces (Corollary~7.3).

\section{Local tracial $\mathcal{C}$-algebra}

We recall the definition of local $\mathcal{C}$-property in 
\cite{OP:Rohlin}.

\begin{dfn}\label{D:FSat}
Let ${\mathcal{C}}$ be a class of separable unital C*-algebras.
Then ${\mathcal{C}}$ is {\emph{finitely saturated}}
if the following closure conditions hold:
\begin{enumerate}
\item\label{D:FSat:1}
If $A \in {\mathcal{C}}$ and $B \cong A,$ then $B \in {\mathcal{C}}.$
\item\label{D:FSat:2}
If $A_1, A_2, \ldots, A_n \in {\mathcal{C}}$ then
$\bigoplus_{k=1}^n A_k \in {\mathcal{C}}.$
\item\label{D:FSat:3}
If $A \in {\mathcal{C}}$ and $n \in \N,$
then $M_n (A) \in {\mathcal{C}}.$
\item\label{D:FSat:4}
If $A \in {\mathcal{C}}$ and $p \in A$ is a nonzero projection,
then $p A p \in {\mathcal{C}}.$
\end{enumerate}
Moreover,
the {\emph{finite saturation}} of a class ${\mathcal{C}}$ is the
smallest finitely saturated class which contains ${\mathcal{C}}.$
\end{dfn}

\vskip 3mm

\begin{dfn}\label{D:LC}
Let ${\mathcal{C}}$ be a class of separable unital C*-algebras.
A {\emph{unital local ${\mathcal{C}}$-algebra}}
is a separable unital C*-algebra $A$
such that for every finite set $S \subset A$ and every $\ep > 0,$
there is a C*-algebra $B$ in the finite saturation of ${\mathcal{C}}$
and a unital *-homomorphism $\ph \colon B \to A$
(not necessarily injective)
such that $\dist (a, \, \ph (B)) < \ep$ for all $a \in S.$
\end{dfn}

\vskip 3mm

When $B$ in Definition~\ref{D:LC} is non unital, we perturb the condition as follows:

\vskip 3mm

\begin{dfn}\label{D:TC}
Let ${\mathcal{C}}$ be a class of separable unital C*-algebras. 
\begin{itemize}
\item[(i)]
A unital C*-algebra $A$ is said to be a unital local tracial $\mathcal{C}$-algebra 
(\cite{YF}) if 
for every finite set $\mathcal{F} \subset A$ and every $\ep > 0,$ and any non-zero $a \in A^+$,
there exist a non-zero projection $p \in A$, a sub C*-algebra $C \subset A$, and 
*-homomorphism  $\varphi \colon C \to A$  
such that $C \in \mathcal{C}$, $\varphi(1_C) = p$, and for all $x \in \mathcal{F}$, 
\begin{enumerate}
\item
$\|xp - px\| < \varepsilon$,
\item
$pxp \in_\varepsilon \varphi(C)$, and
\item
$1 - p$ is Murray-von Neumann equivalent to a projection in $\overline{aAa}$.
\end{enumerate}

\item[(ii)]
A unital C*-algebra $A$ is said to belong to the class TA$\mathcal{C}$ (\cite{EN})  if 
for every finite set $\mathcal{F} \subset A$ and every $\ep > 0,$ and any non-zero $a \in A^+$,
there exist a non-zero projection $p \in A$ and a sub C*-algebra $C \subset A$ 
such that $C \in \mathcal{C}$, $1_C = p$, and for all $x \in \mathcal{F}$, 
\begin{enumerate}
\item
$\|xp - px\| < \varepsilon$,
\item
$pxp \in_\varepsilon C$, and
\item
$1 - p$ is Murray-von Neumann equivalent to a projection in $\overline{aAa}$.
\end{enumerate}
\end{itemize}
\end{dfn}

\vskip 3mm

\begin{rmk}
\begin{enumerate}
\item When a unital C*-algebra $A$ is a unital local tracial $\mathcal{C}$-algebra and 
each $\varphi(C) \in \mathcal{C}$, $A$ belongs to the class TA$\mathcal{C}$.
\item
If $\mathcal{C}$ is the set of finite deimensional C*-algebras $\mathscr F$,
and  the class of interval algebras $\mathscr I$, respectively, then  a local TA${\mathscr F}$-algebra 
belongs to the class of C*-algebras of tracially AF C*-algebras (\cite{Lin:tracial}), and a local
TA${\mathscr I}$-algebra belongs to the class of tracial topological one (TAI- algebras) 
(\cite{Lin:tracial one}) in the sense of Lin, respectively.
\end{enumerate}
\end{rmk}

\vskip 3mm

We have the following relation between the local $\mathcal{C}$ property and the local
tracial approximational $\mathcal{C}$ property.

Recall that a C*-algebra $A$ is said to have the {\it Property (SP)}
if any nonzero hereditary C*-subalgebra of $A$ has a nonzero projection.

\vskip 1mm 

\begin{prp}\label{P:Local C vs Tracial}
Let $\mathcal{C}$ be a finite saturated set and 
$A$ be a local tracial $\mathcal{C}$-algebra. 
Then $A$ has the property (SP) or $A$ is a local $\mathcal{C}$ algebra.
\end{prp}

\vskip 3mm

\begin{proof}
Suppose that $A$ does not have the Property (SP). Then there is a positive element $a \in A$ 
such that $\overline{aAa}$ has no non zero projection. 
Since $A$ is a local TA$\mathcal{C}$-algebra, 
for every finite set $\mathcal{F} \subset A$ and every $\ep > 0,$ 
we conclude that there are a unital C*-algebra $C$ in the class $\mathcal{C}$ and  
a unital *-homomorphism $\varphi \colon C \to A$ such that $\mathcal{F}$ can be approximated 
by a C*-algebra $\varphi(C)$ to whithin $\varepsilon$.
 Hence $A$ is a local $\mathcal{C}$-algebra.
\end{proof}

\vskip 3mm

\section{Tracial Rokhlin property for finite group actions}

Inspired by the concept of the tracial AF C*-algebras in \cite{Lin:tracial}
Phillips defined the tracial Rokhlin property for a finite group action in 
\cite[Lemma 1.16]{Phillips:tracial} as follows:

\vskip 1mm

\begin{dfn} Let $\alpha$ be the action of a finite group $G$ on a unital 
an infinite dimensional finite simple separable unital C*-algebra $A$. 
An $\alpha$ is said to have the {\it tracial Rokhlin property} if for every finite set 
$F \subset A$, every $\varepsilon > 0$, and every nonzero positive $x \in A$, 
there are mutually orthogonal projections $e_g \in A$ for $g \in G$ such that:
\begin{enumerate}
\item $\|\alpha_g(e_h) - e_{gh}\| < \varepsilon$ for all $g, h \in G$.
\item $\|e_ga - ae_g\| < \varepsilon$ for all $g \in G$ and all $a \in F$.
\item With $e = \sum_{g\in G}e_g$, the projection $1 - e$ is Murray-von Neumann 
equivalent to a projection in the hereditary subalgebra of $A$ generated by $x$.
\end{enumerate}
\end{dfn}

\vskip 3mm

It is obvious that the Rokhlin property is stronger than the 
tracial Rokhlin property. 
As pointed in \cite{Izumi:Rohlin1} the Rokhlin property gives rise to several 
$K$-theoretical constrains. For example there is no action 
with the Rokhlin property on the noncommutative 2-torus. 
On the contrary, if $A$ is a simple higher dimensional noncommutative 
torus, which standard unitary generators $u_1, u_2, \dots, u_n$,  then
the automorphism which sends $u_k$ to $\exp(2\pi/n)u_k$, and fixes $u_j$ 
for $j \not= k$,  generates an action $\Z/n\Z$ and has the tracial Rokhlin property, 
but for $n > 1$ never has the Rokhlin property (\cite{Phillips:tracial}).

\vskip 3mm

\begin{lem}\label{Basiclemma}(\cite{Phillips:tracial}, \cite{DA})
Let $a$ be an infinite dimensional stably finite simple C*-algebra with 
the Property SP. Let  $G$ be a finite group of order $n$
and let $\alpha\colon G \rightarrow \Aut(A)$ be an action of $G$ 
with the tracial Rokhlin property. 
Then for any $\varepsilon > 0$, every $N \in \N$, 
any finite set $\mathcal{F} \subset A \rtimes_\alpha G$, and any non-zero 
$z \in (A \rtimes_\alpha G)^+$, 
there exist a non-zero projection $e \in A \subset A \rtimes_\alpha G$, 
a unital C*-subalgebra $D \subset e(A \rtimes_\alpha G)e$, a projection $f \in A$ and 
an isomorphism $\phi \colon M_n \otimes fAf \rightarrow D$, such that 
\begin{itemize}
\item For every $a \in \mathcal{F}$ $\|ea - ae \| < n\varepsilon$.
\item With $(e_{gh})$ for $g, h \in G$ being a system of matrix units 
for $M_n$, we have $\phi(e_{11} \otimes a) = a$ for all $a \in fAf$ 
and $\phi(e_{gg} \otimes 1) \in A$ for $g \in G$.
\item With $(e_{gg})$ as in (1), we have 
$\|\phi(e_{gg} \otimes a) - \alpha_g(a)\| \leq \varepsilon \|a\|$ 
for all $a \in fAf$.
\item For every $a \in F$ there exist $b_1, b_2 \in D$ such that 
$\|ea - b_1\| < \varepsilon$, $\|ae - b_2\| < \varepsilon$ and 
$\|b_1\|, \|b_2\| \leq \|a\|$.
\item
$e = \sum_{g\in G}\phi(e_{gg} \otimes 1)$. 
\item
$1 - e$ is Murray-von Neumann equivalent to a projection in 
$\overline{z(A \rtimes_\alpha G)z}$.
\end{itemize}
\end{lem}

\begin{proof}
This comes from the same argument in Lemma 3.1 of \cite{DA}.
\end{proof}

\vskip 3mm

\begin{thm}\label{Th:traciality}
Let $\mathcal{C}$ be a class of infinite dimensional 
stably finite separable unital C*-algebras which 
is closed under the following conditions:
\begin{enumerate}
\item[$(1)$]
$A \in {\mathcal C}$ and $B \cong A$, then $B \in \mathcal{C}$.
\item[$(2)$]
If $A \in \mathcal{C}$ and $n \in \N$, then $M_n(A) \in \mathcal{C}$.
\item[$(3)$] 
If $A \in \mathcal{C}$ and $p \in A$ is a nonzero projection, 
then $pAp \in \mathcal{C}$.
\end{enumerate}

For any simple C*-algebra $A \in TA{\mathcal C}$
and an action $\alpha$ of a finite group $G$ if $\alpha$ has the tracial 
Rokhlin property, then the crossed product algebra $A \rtimes_\alpha G$ 
belongs to the class $TA{\mathcal C}$.
\end{thm}

\begin{proof}
Since $\alpha$ is outer by Lemma~1.5 of \cite{Phillips:tracial}, 
$A \rtimes_\alpha G$ is simple by Theorem 3.1 of \cite{Kishimoto:81}.

By Lemma~1.13 of \cite{Phillips:tracial}, $A$ has the Property (SP) or 
$\alpha$ has the strictly Rokhlin property.

If $\alpha$ has the strictly Rokhlin property, then 
$A \rtimes_\alpha G \in TA\mathcal{C}$ from the argument 
in the proof of Proposition~1.7 in \cite{OP:Rohlin}.

Let $\mathcal{F} \subset A \rtimes_\alpha G$ be a finite set with 
$\|z \| \leq 1$ and let $\varepsilon > 0$.

Next, suppose that $A$ has the property SP. Then 
there exists a non-zero projection $q \in A$ which 
is Murray-von Neumann equivalent in $A \rtimes_\alpha G$ 
to a projection in $\overline{z(A \rtimes_\alpha G)z}$ by 
Theorem~2.1 of \cite{O:SP-property}.
Since $A$ is simple, take orthogonal projections 
$q_1, q_2$ with $q_1, q_2 \leq q$ by \cite{Lin:book}.
Set $n = card(G)$, and set $\varepsilon_0 = \frac{1}{16n}\varepsilon$. 

By Lemma~\ref{Basiclemma} for $n$ as given, 
for $\varepsilon_0$ in place of $\varepsilon$, and  for 
$q_1$ in place of $z$
there exist a non-zero projection $e \in A \subset A \rtimes_\alpha G$, 
a unital C*-subalgebra $D \subset e(A \rtimes_\alpha G)e$, 
a projection $f \in A$ and 
an isomorphism $\phi \colon M_n \otimes fAf \rightarrow D$, such that 
\begin{enumerate}
\item[(0)] For every $a \in \mathcal{F}$ $\|ea - ae \| < n\varepsilon_0$.
\item[(i)] With $(e_{gh})$ for $g, h \in G$ being a system of matrix units 
for $M_n$, we have $\phi(e_{11} \otimes a) = a$ for all $a \in fAf$ 
and $\phi(e_{gg} \otimes 1) \in A$ for $g \in G$.
\item[(ii)] With $(e_{gg})$ as in (1), we have 
$\|\phi(e_{gg} \otimes a) - \alpha_g(a)\| \leq \varepsilon_0$ 
for all $a \in fAf$.
\item[(iii)]  For every $a \in F$ there exist $d_1, d_2 \in D$ such that 
$\|ea - d_1\| < \varepsilon_0$, 
$\|ae - d_2\| < \varepsilon_0$ and 
$\|d_1\|, \|d_2\| \leq 1$.
\item[(iv)]
$e = \sum_{g\in G}\phi(e_{gg} \otimes 1)$. 
\item[(v)] $1 - e$ is Murray-von Neumann equivalent to a projection in 
$\overline{q_1(A \rtimes_\alpha G)q_1}$.
\end{enumerate}

We note that there is a finite set $T$ in the closed ball of 
$M_n \otimes fAf$ such that for every $a \in F$ there are 
$b_1, b_2 \in T$ such that $\|ea - \phi(b_1)\| < \varepsilon_0$
and $\|ae - \phi(b_2)\| < \varepsilon_0$.

Since $A$ is simple, we 
choose equivalent nonzero projections 
$f_1, f_2 \in A$ such that $f_1 \preceq f$ and $f_2 \preceq q_2$ 
by \cite{Lin:book}. 
Since $M_n \otimes fAf \in TA\mathcal{C}$ by Lemma 2.3 of \cite{EN},
there is a projection $p_0 \in M_n \otimes fAf$ 
and a C*-subalgebra $C \subset M_n \otimes fAf$ such that 
$C \in \mathcal{C}$, $1_C = p_0$ such that 
$\|p_0b - bp_0\| < \frac{1}{4}\varepsilon$ for all $b \in T$,
such that for every $b \in T$ 
there exists $c \in C$ with $\|p_0bp_0 - c\| < \frac{1}{4}\varepsilon$, 
and such that 
$1 - p_0 \preceq e_{11} \otimes f_1$ in 
$M_n \otimes fAf$.

Set $p = \phi(p_0)$, and set $E = \phi(C)$, which 
is a unital subalgebra of $p(A \rtimes_\alpha G)p$ and belongs to $\mathcal{C}$. 
Note that $e - p = \phi(1 - p_0) \preceq \phi(e_{11} \otimes f_1) = f_1$.

Let $a \in \mathcal{F}$. Choose $b \in T$ such that 
$\|\phi(b) - eae\| < \varepsilon_0 < \frac{1}{4}\varepsilon$. 
Then, using $pe = ep = p$,
\begin{align*}
\|pa - ap\| &\leq 2\|ea - ae\| + \|peae - eaep\|\\
&\leq 2\|ea - ae\| + 2\|eae - \phi(b)\| + \|p_0b - bp_0\|\\
&< 2n\varepsilon_0 + 2\varepsilon_0 + \varepsilon_0\\
&= (2n + 3)\varepsilon_0 < \varepsilon.
\end{align*}

Choosing $c \in C$ such that $\|p_0bp_0 - c\| < \frac{1}{4}\varepsilon$, 
the element $\phi(c)$ is in $E$ and satisfies
\begin{align*}
\|pap - \phi(c)\| &= \|peaep - \phi(c)\|\\
&= \|p(eae - \phi(b))p + p \phi(b)p - \phi(c)\|\\
&\leq \|eae - \phi(b)\| + \|p_0bp_0 - c\| \\
&< \frac{1}{4}\varepsilon + \frac{1}{4}\varepsilon < \varepsilon.
\end{align*}
Finally, in $A \rtimes_\alpha G$ we have 
\begin{align*}
1 - p = (1 - e) + (e - p) \preceq q_1 + q_2 \leq q
\end{align*}
and $q$ is Murray-von Neumann equivalent to a projection in 
$\overline{z(A \rtimes_\alpha G)z}$. 
\end{proof}

\vskip 3mm

\begin{thm}\label{Th:stablerank}(\cite{FF})
Let $A$ be an infinite dimensional  simple separable unital C*-algebra 
with stable rank one and let
$\alpha \colon G \rightarrow \Aut(A)$ be an  
action of a finite group $G$ with the tracial Rokhlin property. 
Then $A \rtimes_\alpha G$ has stable rank one.
\end{thm}

\begin{proof}
Let $\mathcal{C}$ be the set of unital C*-algebras with stable rank one.
Then $\mathcal{C}$ is closed under three conditions 
in Theorem~\ref{Th:traciality}. Then from Theorem~\ref{Th:traciality}
$A \rtimes_\alpha G$ belongs to the class TA$\mathcal{C}$.

Hence from Therem~4.3 of \cite{EN} $A \rtimes_\alpha G$ has 
stable rank one.
\end{proof}

\vskip 5mm

We can also show that when $\mathcal{C}$ is the class of unital separable C*-algberas with 
real rank zero and $A$ is a simple unital C*-algebra in the class TA$\mathcal{C}$, $A$ has real rank zero.

\vskip 5mm

\begin{thm}\label{Th:realrank}
Let $\mathcal{C}$ be the class of unital separable C*-algberas with 
real rank zero. Then any simple unital C*-algebra in the class TA$\mathcal{C}$ has real rank zero.
\end{thm}

\begin{proof}
We can deduce it from the same argument in the proof of Theorem~4.3 in \cite{EN}.
\end{proof}

\vskip 5mm

\begin{cor}\label{Cor:realrank}(\cite{FF})
Let $A$ be an infinite dimensional  simple separable unital C*-algebra 
with real rank zero and let
$\alpha \colon G \rightarrow \Aut(A)$ be an  
action of a finite group $G$ with the tracial Rokhlin property. 
Then $A \rtimes_\alpha G$ has real rank zero.
\end{cor}

\begin{proof}
Let $\mathcal{C}$ be the set of unital C*-algebras with real rank zero.
Then $\mathcal{C}$ is closed under three conditions 
in Theorem~\ref{Th:traciality}. Then from Theorem~\ref{Th:traciality}
$A \rtimes_\alpha G$ belongs to the class TA$\mathcal{C}$.

Hence from Theorem~\ref{Th:realrank} $A \rtimes_\alpha G$ has 
real rank zero.
\end{proof}

\vskip 5mm

\begin{thm}\label{Th:comparison}
Let $A$ be an infinite dimensional  simple separable unital C*-algebra 
such that the order on projections over $A$ is determined by traces
and let
$\alpha \colon G \rightarrow \Aut(A)$ be an 
action of a finite group $G$ with tracial Rokhlin property. 
Then  the order on projections over $A \rtimes_\alpha G$ 
is determined by traces.
\end{thm}

\begin{proof}
Let $\mathcal{C}$ be the set of unital C*-algebras
such that the order on projections over them is determined by traces.
Then $\mathcal{C}$ is closed under three conditions 
in Theorem~\ref{Th:traciality}. Then from Theorem~\ref{Th:traciality}
$A \rtimes_\alpha G$ belongs to the class TA$\mathcal{C}$.

Hence from Therem~4.12 of \cite{EN} 
the order on projections over $A \rtimes_\alpha G$ 
is determined by traces.
\end{proof}

\vskip 3mm

\begin{dfn}\label{Df:tracialtoporogicalrank}
Let $\mathcal{T}^{(0)}$ be the calss of all finite dimensional 
C*-algebras and let $\mathcal{T}^{(k)}$ be the class of all 
C*-algebras with the form $pM_n(C(X))p$, where $X$ is a finite 
CW complex with dimension $k$ and $p \in M_n(C(X))$ is a projection.

A simple unital C*-algebra $A$ is said to have tracial 
topological rank no more than $k$ if for any set 
$\mathcal{F} \subset A$, and $\varepsilon > 0$ and any nonzero positive 
element $a \in A$, there exists a C*-subalgebra $B \subset A$ with 
$B \in \mathcal{T}^{(k)}$ and $id_B = p$ such that 
\begin{enumerate}
\item[$(1)$]
$\|xp - px\| < \varepsilon$ \\
\item[$(2)$]
$pxp \in_\varepsilon C$, and\\
\item[$(3)$]
$1 - p$ is Murray-von Neumann equivalent to a projection 
in $\overline{aAa}$.
\end{enumerate}
\end{dfn}

\vskip 3mm

\begin{thm}\label{Th:tracialtopologicalrank}(\cite{OP:tracial Rokhlin})
Let $A$ be an infinite dimensional simple unital C*-algebra with 
tracial topological rank no more than or equal to $k$, and 
$\alpha \colon G \rightarrow \Aut(A)$ be an 
action of a finite group $G$ with tracial Rokhlin property. 
Then
$A \rtimes_\alpha G$ has tracial topological rank more than or 
equal to $k$.
\end{thm}

\begin{proof}
Let $\mathcal{C}$ be the set $\mathcal{T}^{(k)}$.
Then $\mathcal{T}^{(k)}$ is closed under three conditions 
in Theorem~\ref{Th:traciality}. Then from Theorem~\ref{Th:traciality}
$A \rtimes_\alpha G$ belongs to the class TA$\mathcal{C}$.
This means that $A \rtimes_\alpha G$ has tracial topological rank no
more that or 
equal to $k$ from the Definition~\ref{Df:tracialtoporogicalrank}.
\end{proof}

\vskip 5mm


\section{Tracial Rokhlin property for a inclusion of unital C*-algebras}

Let $P \subset A$ be a inclusion of unital C*-algebras and 
$E \colon A \rightarrow P$ be a conditional expectation of 
index finite.

As in the case of the Rokhlin property in \cite{KOT} we can define the tracial Rokhlin property 
for a conditional expectation for an inclusion of unital C*-algebras.

\begin{dfn}\label{tracial Rokhlin}
Let $P \subset A$ be an inclusion of unital C*-algebras and $E\colon A \rightarrow P$ be 
a conditional expectation of index finite. 
A conditional expectation $E$ is said to have the {\it tracial Rokhlin property} 
if for any nonzero positive $z \in A^\infty$ there exists a  projection $e \in A' \cap A^\infty$ satisfying 
$$
({\Index}E)E^\infty(e) = g
$$
is a projection and $1 - g$ is Murray-von Neumann equivalent to a projection in the hereditary subalgebra 
of $A^\infty$ generated by $z$, 
and a map $A \ni x \mapsto xe$ is injective. We call $e$ a Rokhlin projection.
\end{dfn}

\vskip 3mm

As in the case of an action with the tracial Roklin property (\cite[Lemma~1.13]{Phillips:tracial}) 
if $E\colon A \rightarrow P$ is a conditional expectation of index finite type for an 
inclusion of unital C*-algebras $P \subset A$ and $E$
has the tracial Rokhlin property, 
then $A$ has Property (SP) or $E$ has the Rokhlin property, 
that is, there is Rokhlin projection 
$e \in A' \cap A^\infty$ such that $(\mathrm{Index}E)E^\infty(e) = 1$.

\vskip 1mm

\begin{lem}\label{lem:trcial and SP}
Let $P \subset A$ be an inclusion of unital C*-algebras and $E\colon A \rightarrow P$
be a conditional expectation of index finite type. 
Suppose that $E$ has the tracial Rokhlin property,
then $A$ has the Property (SP) or $E$ has the Rokhlin property.
\end{lem}

\vskip 3mm

\begin{proof}
If $A$ does not have the Property (SP), then $A^\infty$ does not have Property (SP), that is, 
there is a nonzero positive element $x \in A^\infty$ which generates a hereditary subalgebra 
which contains no nonzero projection. Since $E$ has the tracial Rokhlin property, there 
exists a projection $e \in A_\infty$ such that $1 - ({\Index}E)E^\infty(e)$ is 
equivalent to some projection in $\overline{xA^\infty x}$. 
Hence $1 - ({\Index}E)E^\infty(e) = 0$. This implies that $E$ has the Rokhlin property.
\end{proof}

\vskip 3mm

\begin{rmk}\label{rmk:sp}
\begin{enumerate}
\item
When $A$ has the Property (SP), we know that 
$g$ in Definition~\ref{tracial Rokhlin} is not zero. 
Indeed, since $A^\infty$ has the Property (SP), 
from \cite[Lemma~1.15]{Phillips:tracial}
for a positive element $x \in A^\infty$ with $\|x\| = 1$ and $\varepsilon > 0$ 
there exists a nonzero projection $q \in \overline{xA^\infty x}$ such that,  
whenever $1 - g \preceq q \in \overline{xA^\infty x}$, then 
$$
\|gxg\| > 1 - \varepsilon.
$$
This implies that $g \not= 0$.
\item
We always have $g \in P' \cap P^\infty$. Indeed, 
for any $x \in P$
\begin{align*}
xg &= x({\Index}E)E^\infty(e)\\
&= ({\Index}E)E^\infty(xe)\\
&= ({\Index}E)E^\infty(ex)\\
&= ({\Index}E)E^\infty(e)x\\
&= gx.
\end{align*}
\end{enumerate}
\end{rmk}

\vskip 3mm

\begin{rmk}\label{rmk:simple}
In Definition~\ref{tracial Rokhlin} when $A$ is simple,
\begin{enumerate}
\item
we do not need the injectivity of the map $A \ni x \mapsto xe$.
\item
we have $ege = e$. Indeed, since $A$ is simple, ${\Index}E$ is scalar 
by \cite[Remark~2.3.6]{Watatani:index} and 
from \cite[Lemma~2.1.5 (2)]{Watatani:index} 
we have
$$
E(e) \geq \frac{1}{({\Index}E)^2}e,
$$
hence $g \geq \frac{1}{{\Index}E}e$. 
Then 
\begin{align*}
(1 - g)\frac{1}{{\Index E}}e(1-g) &= 0\\
(1 - g)e &= 0\\
e &= ge = ege
\end{align*}
\item
Let $E\colon A \rightarrow P$ be of index finite with the tracial Rokhlin 
property and let consider the basic extension
$$
P \subset A \subset B.
$$
Then the Rokhlin projection $e \in A' \cap A^\infty$ satisfies 
$eBe = Ae$. 

Indeed, 
let $e_p$ be the Jones projection for the inclusion 
$A \supset P$. Set $f = {\Index E}ee_pe$. Then 
\begin{align*}
f^2 &= ({\Index E})^2ee_pee_pe\\
&= ({\Index E})^2eE^\infty(e)e_pe\\
&= ({\Index E})^2eE^\infty(e)e_pe\\
&= ({\Index E})ege_pe\\
&= ({\Index E})ee_pe\\
&= f.
\end{align*}
Since 
\begin{align*}
\hat{E}^\infty(e - f) &= e - {\Index}E\hat{E}^\infty(ee_pe)\\
&= e - e = 0,
\end{align*}
we have $e = f = {\Index E}ee_pe$, 
$\hat{E}$ is the dual conditional expectation for $E$.

Then we have for any $x, y \in A$
\begin{align*}
e(xe_py)e &= xee_pey\\
&= ({\rm Index}E)^{-1}xey\\
&= ({\rm Index}E)^{-1}xye \in Ae.
\end{align*}
Hence $eBe \subset Ae$. Conversely, since $A \subset B$, $Ae \subset eBe$, and 
we conclude that $eBe = Ae$.
\hfill$\qed$
\end{enumerate}
\end{rmk}

\vskip 3mm

The following is the heredity of the Property SP for an inclusion of 
unital C*-algebras.

\vskip 3mm

\begin{prp}\label{prp:SP-property}
Let $P \subset A$ be an inclusion of unital C*-algebras with index finite type.
Suppose that $A$ is simple and $E\colon A \rightarrow P$ 
has the tracial Rokhlin property. Then we have 
\begin{enumerate}
\item $P$ is simple.
\item $A$ has the Property (SP) if and only if $P$ has the Property (SP).
\end{enumerate}
\end{prp}

\vskip 3mm

\begin{proof}
$(1)$: 
Let $e$ be a Rokhlin projection for $E$ and $P \subset A \subset B$ be 
the basic extension.   Since $P$ is stably isomorphic to 
$B$, we will show that $B$ is simple, that is, $B' \cap B = \C$.

Since $e = [(e_n)] \in A' \cap A^\infty$, for any $x \in A' \cap B$ we have 
$$
ex = [(e_n)]x = [(e_nx)] = [(xe_n)] = xe.
$$

We may assume that $x = a_1e_pa_2$, where $e_p$ is the Jones projection for $E$.
Then 
\begin{align*}
xe &= exe\\
&= e(a_1e_pa_2)e\\
&= a_1ee_pea_2 \\
&= ({\rm Index}E)^{-1}a_1a_2e\\
&= \hat{E}(x)e,
\end{align*}
where $\hat{E} \colon B \rightarrow A$ be the dual conditional expectation of $E$.
Note that $\hat{E}(x) \in A'$. Hence we have $xe \in (A' \cap A)e$.

Since $A$ is simple  and $(B' \cap B)e \subset (A'\cap B)e \subset (A' \cap A)e$,
we have $(B' \cap B)e = \C e$. 
Since the map $\rho\colon A' \cap B \rightarrow (A' \cap B)e$ by $\rho(x) = xe$ is 
an isomorphism, $B' \cap B = \C$, that is, $B$ is simple, and $P$ is simple.

$(2)$: It follows from \cite{O:SP-property}.
\end{proof}

\vskip 3mm

\begin{prp}\label{prp:tracial for finite group}
Let $G$ be a finite group, $\alpha$ an action of $G$ on 
an infinite dimensional finite simple separable unital C*-algebra $A$, 
and $E$ the canonical conditional expectation from $A$ onto the fixed 
point algebra $P = A^\alpha$ 
defined by 
$$
E(x) = \frac{1}{|G|}\sum_{g\in G}\alpha_g(x) \quad \text{for} \ x \in A, 
$$
where $|G|$ is the order of $G$. 
Then $\alpha$ has the tracial Rokhlin property 
if and only if $E$ has the tracial Rokhlin property.
\end{prp}

\vskip 3mm

\begin{proof}
Suppose that $\alpha$ has the tracial Rokhlin property. 
Since $A$ is separable, there is an increasing sequence of finite sets 
$\{F_n\}_{n\in\N} \subset A$ such that $\overline{\cup_{n\in\N}F_n} = A$.
Let any nonzero positive element $x = (x_n) \in A^\infty$. 
Then we may assume that each $x_n$ is nonzero positive element.
The simplicity of $A$ implies that the map $A \ni x \mapsto xe$ is 
injective.

Since $\alpha$ has the tracial Rokhlin property, for each $n$ 
there are mutually orthogonal projections $\{e_{g,n}\}_{g \in G}$
such that 
$$
\|\alpha_h(e_{g,n}) - e_{hg,n}\| < \frac{1}{n}
$$ for all $g, h \in G$,
$$
\|[e_{g,n}, a]\| < \frac{1}{n}
$$
for all $g\in G$, and $a \in F_n$ with $\|a\| \leq 1$
$$
1 - \sum_{g\in G}e_{g,n}$$ is equivalent to a nonzero projection $q_n$ 
in $\overline{x_nAx_n}$.

Set $e_g = [(e_{g,n})] \in A^\infty$ for $g \in G$. 
Then for all $g, h \in G$
\begin{align*}
\|\alpha^\infty_h(e_g) - e_{hg}\| &= \lim\sup\|\alpha_h(e_{g,n}) - e_{hg,n}\| = 0,
\end{align*}
hence $\alpha^\infty_h(e_g) = e_{hg}$ for all $g, h \in G$.

For all $a \in \cup_{n\in\N}F_n$ with $\|a\| \leq 1$ and all $g \in G$
we have 
\begin{align*}
\|[e_g,a]\| &= \lim\sup\|[e_{g,n},a]\| = 0,
\end{align*}
hence $e_g \in A_\infty$ for all $g \in G$.

Set $q = (q_n) \in A^\infty$. Then $q$ is nonzero projection in $\overline{xA^\infty x}$
and 
\begin{align*}
1 - \sum_{g \in G}e_g &= (1 - \sum_{g\in G}e_{g,n}) \sim (q_n) = q.
\end{align*}

Therefore, if we set $e = e_1$ for the identity element $1$ in $G$, 
then $e \in A' \cap A^\infty$ and 
\begin{align*}
E^\infty(e) &= \frac{1}{|G|}\sum_{g\in G}\alpha^\infty_g(e)\\
&= \frac{1}{|G|}\sum_{g\in G}e_g
\end{align*}
and 
\begin{align*}
1 - |G|E^\infty(e) &= 1 - \sum_{g\in G}e_g \sim q \in \overline{xA^\infty x}.
\end{align*}
It follows that $E$ has the tracial Rokhlin property.

Conversely, suppose that $E$ has the tracial Rokhlin property. 
From Lemma~\ref{lem:trcial and SP} $A$ has the Propety (SP) or 
$E$ has the Rokhlin property.
If $E$ has the Rokhlin property, then $\alpha$ has the Rokhlin property by 
Proposition~3.2 in \cite{KOT}, hence $\alpha$ has the tracial Rokhlin property 
from the definition.

We may assume that $A$ has Property (SP).
Then for any finite set $F \subset A$, 
$\varepsilon > 0$, and any nonzero positive element $x \in A$  
there is a projection $e \in A_\infty$ such that $|G| E^\infty (e)  (= g)$ is a projection 
and $1 - g$ is equivalent to a projection $q \in \overline{xA^\infty x}$.
We note that $g \not= 0$ by Remark~\ref{rmk:sp}, and $e \not= 0$.
When we write $e = (e_n)$ and $q = (q_n)$, we may assume that for each $n \in \N$
$e_n$ is projection and $1 - e_n$  is equivalent to $q_n$.

Define $e_g = \alpha_g^\infty(e)\in A_\infty$ for $g \in G$, 
write $e_g = [(\alpha_g(e_n))] = [(e_{g,n})]$ for $g \in G$. 
Then since we have 
\begin{align*}
\sum_{g\in G}e_g &= \sum_{g\in G}\alpha_g(e) = |G|E^\infty(e) = g
\end{align*}
and $g$ is projection, we may assume that  that  $\{e_{g,n}\}_{g \in G}$
are mutually orthogonal projections for each $n \in \N$ by 
\cite[Lemma~2.5.6]{Lin:book}.

Then 
$\alpha_h^\infty(e_g) = e_{hg}$ for all $g, h \in G$, 
$\|[e_g, a]\| = 0$ for all $a \in F$ 
and all $g \in G$, and 
\begin{align*}
1 - \sum_{g \in G}e_g &= 1 - \sum_{g \in G}\alpha^\infty_g(e)\\
&=1 - |G| E(e) \\
&= 1 - g \sim q \in \overline{xA^\infty x},
\end{align*}

Then there exists $n \in \N$ such that 
$$
\|\alpha_h(e_{g,n}) - e_{hg, n}\| < \varepsilon
$$
for all $g, h \in G$,
$$
\|[e_{g,n}, a]\| < \varepsilon
$$
for all $a \in F$ and $g \in G$, and 
$$
1 - \sum_{g\in G}e_{g,n} \sim q_n \in \overline{xAx}.
$$

Set $f_g = e_{g,n}$ for $g \in G$, then we have 

$$
\|\alpha_h(f_g) - f_{hg}\| < \varepsilon,
$$
for all $g, h \in G$,
$$
\|[f_g, a]\| < \varepsilon
$$
for all $a \in F$ and $g \in G$, and 
$$
1 - \sum_{g\in G}f_g \sim q_n \in \overline{xAx}.
$$

Hence $\alpha$ has the tracial Rokhlin property.
\end{proof}

\vskip 3mm

The following is a key lemma to prove the main theorem in this section.

\vskip 3mm

\begin{lem}\label{lem:embedding for tracial}
Let $A \supset P$ be an inclusion of unital \ca s and 
$E$  a conditional expectation from $A$ onto $P$ with index finite type.
Suppose that $A$ is simple.
If $E$ has the tracial Rokhlin property with a Rokhlin projection $e \in A_\infty$
and a projection $g = ({\Index}E)E^\infty(e)$, then 
there is a unital linear map 
$\beta \colon A^\infty 
\rightarrow P^\infty g$ such that 
for any $x \in A^\infty$ there exists 
the unique element $y$ of $P^\infty$ such that $xe = ye = \beta(x)e$ and 
$\beta(A' \cap A^\infty) \subset P' \cap P^\infty g$. 
In particular, $\beta_{|_A}$ is a unital injective *-homomorphism and 
$\beta(x) = xg $ for all  $x \in P$.
\end{lem}

\vskip 3mm

\begin{proof}
Since $E$ has the tracial Rokhlin property, $A$ has the Property (SP) or 
$E$ has the Rokhlin property by Lemma~\ref{lem:trcial and SP}. 
If $E$ has the Rokhlin property, then the conclusion comes from 
Lemma~2.5 in \cite{OT} with $g = 1$. Therefore we may assume that 
$A$ has the Property (SP).

Since $A$ has the Property (SP), $g$ and $e$ are nonzero projections 
by Remark~\ref{rmk:sp}. 
As in  the same argument in the proof of Lemma~2.5 in \cite{OT}
we have for any element $x$ in $A^\infty$ there exists a unique element 
$y = ({\Index}E)E^\infty(xe) \in P^\infty$ such that $xe = ye$. 
Note that since $eg = e$ by Remark~\ref{rmk:simple}, we have
\begin{align*}
yg &=  ({\Index}E)E^\infty(xe)g \\
&= ({\Index}E)E^\infty(xeg)\\
&= ({\Index}E)E^\infty(xe) = y.
\end{align*}

Then we can define a unital map $\beta \colon A^\infty \rightarrow P^\infty g$ 
such that $xe = ye = \beta(x)e$ and $\beta(A' \cap A^\infty) \subset P' \cap P^\infty g$. 

Note that $\beta$ is injective. Indeed,  if $\beta(x) = 0$ for $x \in A$ xe = 0. 
Hence from the definition of the tracial Rokhlin property for $E$, $x = 0$.

Since for any $x \in A$
\begin{align*}
\beta(x)g &= ({\Index}E)E^\infty(xe)g\\
&= ({\Index}E)E^\infty(xeg)\\
&= ({\Index}E)E^\infty(ex) \ (= \beta(x))\\
&= ({\Index}E)E^\infty(gex)\\
&= g({\Index}E)E^\infty(xe)\\
&= g\beta(x),
\end{align*}
we know that 
$\beta_{|A}$ is a unital *-homomorphism from $A$ to $gP^\infty g$ 
from the same argument as in the proof of Lemma~2.5 in \cite{OT}.
In particular for any $x \in P$ we have 
\begin{align*}
\beta(x) &= ({\Index}E)E^\infty(xe)\\
&= x({\Index}E)E^\infty(e)\\
&= xg (= gx).
\end{align*}
\end{proof}

\vskip 3mm

\begin{prp}\label{prp:stableinclusion}
Let $P \subset A$ be an inclusion of unital C*-algebras with index finite type
and $E\colon A \rightarrow P$ has the tracial Rokhlin property. 
Suppose that $A$ is simple with $\tsr(A) = 1$. Then $\tsr(P) = 1$.
\end{prp}

\begin{proof}
Since $E$ has the tracial Rokhlin property, $E$ has the Rokhlin property or 
$A$ has the Property (SP) by Lemma~\ref{lem:trcial and SP}. 
If $E$ has the Rokhlin property, we conclude that $\tsr(P) = 1$ by \cite{KOT}.
Therefore, we assume that $A$ has the Property (SP). Then we know that 
$P$ is simple and has the Property (SP) by Proposition~\ref{prp:SP-property}.

Since $A$ is stably finite and an inclusion $P \subset A$ is of index-finite type, 
$P$ is stably finite. Hence,
using the idea in \cite{Ro:UHF} we have only to show that any two sided zero divisor 
in $P$ is approximated by invertible elements in $P$.

Let $x \in P$ be a two sided zero diviser. From \cite[Lemma~3.5]{Ro:UHF} we may assume that 
there is a positive element $y \in P$ such that $yx = 0 = xy$. 
Since $P$ has the Property (SP), there is a non-zero projection $e \in \overline{yPy}$. 
Since $P$ is simple, we can take orthogonal projections $e_1$ and $e_2$ in $P$ such that 
$e = e_1 + e_2$ and $e_2 \preceq e_1$. Note that $x \in (1 - e_1)A(1 - e_1)$. 
Since $\tsr((1 - e_1)A(1 - e_1)) = 1$, there is an invertible element $b$ in 
$(1 - e_1)A(1 - e_1)$ such that $||x - b|| < \dfrac{1}{3}\varepsilon$.

Since $E$ has the tracial Rokhlin property, there is a projection $g \in P' \cap P^\infty$ 
such that $1 - g \preceq e_2$. That is, there is a partial isometry $w \in P^\infty$ such that
$w^*w = 1 - g$ and $ww^* \leq e_1$. Moreover, $||\beta(x) - \beta(b)|| < \dfrac{1}{3}\varepsilon$ 
by Lemma~\ref{lem:embedding for tracial}. 
Note that $\beta(b)$ is invertible in $gP^\infty g$.

Set 
$$
z = \dfrac{\varepsilon}{3}(e_1 - ww^*) + \dfrac{\varepsilon}{3}w + \dfrac{\varepsilon}{3}w^* + (1 - g)x(1 - g).
$$
Then $z$ is invertible in $e_1P^\infty e_1 + (1 - g)(1 - e_1)P^\infty (1 - e_1)(1 - g)$ and \newline
$||z - (1 - g)x(1 - g)|| < \dfrac{\varepsilon}{3}$.

Then $\beta(b) + z \in P^\infty$ is invertible and 
\begin{align*}
||x - (\beta(b) + z)|| &= ||xg + x(1 - g) - \beta(b) - z||\\
&= ||\beta(x) - \beta(b) + (1 - g)x(1 - g) - z||\\
&\leq ||\beta(x) - \beta(b)|| + ||(1 - g)x(1 - g) - z|| \\
&\leq \frac{\varepsilon}{3} + \frac{\varepsilon}{3} < \varepsilon.
\end{align*}

Write $\beta(b) + z = (y_n)$ such that $y_n$ is invertible in $P$. Therefore, there is a $y_n$ such that 
$||x - y_n|| < \varepsilon$, and we conclude that $\tsr(P) = 1$.
\end{proof}

\vskip 3mm

\begin{prp}\label{prp:realinclusion}
Let $P \subset A$ be an inclusion of unital C*-algebras with index finite type
and $E\colon A \rightarrow P$ has the tracial Rokhlin property. 
Suppose that $A$ is simple with real rank zero. Then $P$ has real rank zero.
\end{prp}

\begin{proof}
Let $x \in P$ be a self-adjoint element and $\varepsilon > 0$.
Consider a continuous real valued function $f$ is defined by $f(x) = 1$ for $|x| \leq \dfrac{\varepsilon}{12}$, 
$f(t) = 0$ if $|x| \geq \dfrac{\varepsilon}{64}$, and $f(t)$ is linear if $\dfrac{\varepsilon}{12} \leq |x| \leq \dfrac{\varepsilon}{6}$.
We may assume that $f(x) \not= 0$. Note that $\|yx\| < \dfrac{\varepsilon}{6}$ for any $y \in \overline{f(x)Pf(x)}$.

Since $A$ is simple and has the Property (SP), $P$ has the Property (SP), that is,   
there is a non-zero projection $e \in \overline{f(x)Pf(x)}$. Moreover, there are orthogonal projections $e_1$ and $e_2$ such that 
$e = e_1 + e_2$ such that $e_2 \preceq e_1$. Then
\begin{align*}
|| x - (1 - e_1)x(1 - e_1)|| &= ||e_1xe_1 + e_1x(1 - e_1) + (1 - e_1)xe_1)|| \\
&< \dfrac{3\varepsilon}{12} = \dfrac{\varepsilon}{4}
\end{align*}

As in the same step in the argument in Proposition~\ref{prp:stableinclusion} we have there is an invertible self-adjoint element $z \in P$ 
such that $|| (1 - e_1)x(1 - e_1) - z|| < \dfrac{2\varepsilon}{3}$. 
Hence, we have $|| x - z|| < \varepsilon$, and we conclude that $P$ has real rank zero.
\end{proof}

The following lemma is important to prove that heredity of the local tracial $\mathcal{C}$-property for 
an inclusion of unital C*-algerbas.

\vskip 3mm

\begin{lem}\label{lem:MNequivalent}
Let $P \subset A$ be an inclusion of unital C*-algebras with index finite type
and $E\colon A \rightarrow P$ has the tracial Rokhlin property.
Suppose that projections $p, q \in P^\infty$ satisfy $ep = pe$ and $q \preceq ep$ in $A^\infty$, where 
$e$ is the Rokhlin projection for $E$.
Then $q \preceq p$ in $P^\infty$.
\end{lem}

\vskip 3mm

\begin{proof}

Let $s$ be a partial isometry in $A^\infty$ such that $s^*s = q$ and 
$ss^* \leq ep$.

Set $v = ({\rm Index} E)^{1/2}E^\infty(s)$. Then 
\begin{align*}
v^*ve_p &= {\rm Index}E E^\infty(s)^*E^\infty(s)e_p\\
&= {\rm Index}E E^\infty(s^*e)E^\infty(es)e_p\\
&= {\rm Index}E e_ps^*ee_pese_p\\
&= e_ps^*ese_p \quad ({\rm Index} E ee_pe = e)\\
&= E^\infty(s^*es)e_p \ \\
&= E^\infty(s^*s)e_p\\
&= E^\infty(q)e_p\\
&= qe_p.
\end{align*}

Hence
\begin{align*}
\hat{E}^\infty(v^*ve_p) &= \hat{E}^\infty(qe_p)\\
({\rm Index}E)^{-1} v^*v &= ({\rm Index} E)^{-1}q\\
v^*v &= q.
\end{align*}

Since 
\begin{align*}
pv &= p ({\rm Index} E)^{1/2} E^\infty(s)\\
&= ({\rm Index} E)^{1/2} E^\infty(ps)\\
&= ({\rm Index} E)^{1/2} E^\infty(s)\\
&= v,
\end{align*}
we have $q \preceq p$ in $P^\infty$.
\end{proof}

\vskip 3mm

\begin{thm}\label{Thm:main theorem}
Let $\mathcal{C}$ be a class of weakly semiprojective C*-algebras 
satisfying  conditions in Theorem~\ref{Th:traciality}. 
Let $A \supset P$ be an inclusion of unital \ca s and 
$E$  a conditional expectation from $A$ onto $P$ with index finite type.
Suppose that $A$ is simple, local tracial $\mathcal{C}$-algebra and
$E$ has the tracial Rokhlin property. Then $P$ is a local tracial $\mathcal{C}$-algebra.
\end{thm}

\begin{proof}
We shall prove that for every finite set $F \subset P$, every 
$\varepsilon > 0$, and $z \in P^+\backslash{0}$ 
there are C*-algebra $Q \in \mathcal{C}$ with 
$q = 1_Q$ and *-homomorphism $\pi\colon Q \rightarrow A$ such that 
$\|\pi(q)x - x\pi(q)\| < \varepsilon$ for all $x \in F$, 
$\pi(q)S\pi(q) \subset_\varepsilon \pi(Q)$, 
and $1 - \pi(q)$ is a equivalent to some non-zero 
projection in $\overline{zPz}$.

Since $E$ has the tracial Rokhlin property, $A$ has the Property (SP) or 
$E$ has the Rokhlin property by Lemma~\ref{lem:trcial and SP}. 

Suppose that  $E$ has the Rokhlin property. 
We have then from Lemma~2.5 in \cite{OT} 
there is a unital *-homomorphism 
$\beta \colon A \rightarrow P^\infty$ such that $\beta(x) = x$ for 
all $x \in P$. Since $A$ is a local tracial $\mathcal{C}$-algebra, there is an algebra $B \in \mathcal{C}$
with $1_B = p$ and *-homomorphism $\pi\colon B \rightarrow A$ such that 
$\|x\pi(p) - \pi(p)x\| < \varepsilon$ for all $x \in F$, 
$\pi(p)F\pi(p) \subset_\varepsilon \pi(B)$, 
and $1 - \pi(p)$ is equivalent to a non-zero projection 
$q \in \overline{zAz}$. Since $E$ has the Rokhlin property, there exists a non-zero 
projection $e \in A' \cap A^\infty$ such that $E^\infty(e) = \frac{1}{{\rm Index E}}$.

Since $B$ is weakly semiprojective, there exists $k \in \N$ 
and $\overline{\beta \circ \pi}\colon B \rightarrow \prod_{n=k}^\infty P$ such that 
$\beta \circ \pi = \pi_k \circ \overline{\beta \circ \pi}$, where 
$\pi_k((b_k, b_{k+1}, \dots)) = (0, \dots, 0, b_k, b_{k+1}, \dots)$.
For each $l \in \N$ with $l \geq k$ let $\beta_l$ a *-homomorphism from 
$B$ to $P$ so that $\overline{\beta \circ \pi}(b) = (\beta_n(b))_{n=k}^\infty$ 
for $b \in B$. Then $\beta \circ \pi(b) = (0, \dots, 0, \beta_k(b), \beta_{k+1}(b), \dots) + C_0(P)$ for $b \in B$ 
and $\beta_l$ is a *-homomorphism for $l \geq k$. 

Since $ 1 - p \sim q \in \overline{zAz}$, 

\begin{align*}
1 - \beta \circ \pi(p) &= 1 - \beta(1 - \pi(p))\\
&= \beta(1 - \pi(p)) \sim \beta \circ \pi(q) \in \beta(\overline{zAz})\\
&[(1 - \beta_k(\pi(p)))] \sim [(q_k)] \in \overline{zP^\infty z},\\
\end{align*}

where each $q_k$ are projections in $P$.
Taking the sufficient large $k$  since $\lim_k\|\beta_k(x) - x\| = 0$ for $x \in P$,
we have  
\begin{enumerate}
\item
$\|x\beta_k(p) - \beta_k(p)x\| < 2\ep$ for any $x \in F$,
\item
$\beta_k(p)F\beta_k(p) \subset_\varepsilon \beta_k(p)\beta_k(B)\beta_k(p)$, and 
\item
$1 - \beta_k(p) = \beta_k(1 - p) \sim q_k \in \overline{zPz}.$
\end{enumerate}
Hence $P$ is a local tracially $\mathcal{C}$-algebra.

Suppose that $A$ has the Property (SP).  
Since $A$ is simple, 
from Proposition~\ref{prp:SP-property} $P$ has also the Property (SP).  
Let $F \subset P$ be a finite set, $\varepsilon > 0$, and $z \in P^+\backslash{0}$.
Since $P$ is simple and has the Property (SP), 
there is orthogonal non-zero projections $r_1, r_2 \in \overline{zPz}$.

Since $A$ is a local tracial $\mathcal{C}$-algebra, there is an algebra $B \in \mathcal{C}$
with $1_B = p$ and *-homomorphism $\pi\colon B \rightarrow A$ 
such that $\|x\pi(p) - \pi(p)x\| < \varepsilon$ for all $x \in F$, 
$\pi(p)F\pi(p) \subset_\varepsilon B$, and $1 - \pi(p)$ 
is equivalent to a non-zero projection $q \in \overline{r_1Ar_1}$. 
Since $E$ has the tracial Rokhlin property, there exist the Rokhlin propjection 
$e' \in A' \cap A^\infty$. 
Take another Rokhlin projection  $e \in A' \cap A^\infty$ for a projection $e'r_2$
such that $g = {\rm Index}E E(e)$ satifies $1 - g$ is equivalent to a projection 
$\overline{e'r_2A^\infty e'r_2}$, that is, $1 - g \preceq e'r_2$ in $A^\infty$.
By Lemma~\ref{lem:MNequivalent} we know, then, that 
$1 - g \preceq r_2$ in $P^\infty$, that is, there is a projection $s \leq r_2 \in P^\infty$ 
such that $1 - g \sim s$. 

Write $g = [(g_n)]$ for some projections $\{g_k\}_{k\in\N} \subset P$. 
From Lemma~\ref{lem:embedding for tracial} there exists injective *-homomorphism 
$\beta\colon A \rightarrow gP^\infty g$ such that $\beta(x) = xg$ for all $x \in P$.
Since $B$ is weakly semiprojective, there exists $k \in \N$ 
and $\overline{\beta \circ \pi}\colon B \rightarrow \prod_{n=k}^\infty P$
such that 
$\beta \circ \pi = \pi_k \circ \overline{\beta \circ \pi}$, where 
$$
\pi_k((b_k, b_k+1, \dots)) = (0, \dots, 0, b_k, b_{k+1}, \dots) + C_0(P).
$$
For each $l \in \N$ with $l \geq k$ let $\beta_l$ be a map from 
$B$ to $g_lPg_l$ so that $\overline{\beta \circ \pi}(b) = (\beta_l(b))_{l=k}^\infty$ 
for $b \in B$. 
Then $\beta \circ \pi(b) = (0, \dots, 0,\beta_k(b), \beta_{k+1}(b), \dots) 
+ C_0(P)$ for all $b \in B$ and $\beta_l$ is a *-homomorphism for $l \geq k$. 

Since $ 1 - \pi(p) \sim q \in \overline{r_1Ar_1}$, 
\begin{align*}
1 - (\beta \circ \pi)(p) &= 1 - g + g - \beta(\pi(p))\\
&= 1 - g + \beta(1 - \pi(p))\\
&\sim s + \beta(q) \in r_2P^\infty r_2 + \iota\circ \beta(\overline{r_1Ar_1})\\
&\subset r_2P^\infty r_2 + r_1P^\infty r_1 \subset \overline{zP^\infty z},\\
\end{align*}
we have $[(1 - \beta_k(p)))] \sim [(q_k)] \in \overline{zP^\infty z}$,\\
where each $q_k$ is projection in $P$.
Taking the sufficient large $k$  since $\lim_k\|\beta_k(x) - x\| = 0$ for $x \in P$,
we have  
\begin{enumerate}
\item
$\|x\beta_k(p) - \beta_k(p)x\| < 2\ep$ for any $x \in F$,
\item
$\beta_k(p)F\beta_k(p) \subset_\varepsilon \beta_k(p)\beta_k(B)\beta_k(p)$, and 
\item
$1 - \beta_k(p) \sim q_k \in \overline{zPz}.$
\end{enumerate}
Hence $P$ is a local tracially $\mathcal{C}$-algebra.
\end{proof}

\vskip 3mm

\begin{cor}\label{Cor:tracialtoporogicalrank}
Let $P \subset A$ be an inclusion of unital \ca s and 
$E$  a conditional expectation from $A$ onto $P$ with index finite type. 
Suppose that $A$ is an infinite dimensional simple C*-algebra with 
tracial topological rank zero (resp. less than or equal to one) and $E$ has the 
tracial Rokhlin property. Then $P$ has tracial rank zero (resp. less than or equal to one).
\end{cor}

\begin{proof}
Since the classes $\mathcal{T}^{(k)}$ \ $(k = 0, 1)$ are semiprojective with respect to a 
class of unital C*-algebras (\cite{Loring:lifting}) and 
finitely saturated (\cite[Examples 2.1 and 2.2, and Lemma~1.6]{OP:Rohlin}), 
the conclusion comes from Theorem~\ref{Thm:main theorem} and Definition~\ref{Df:tracialtoporogicalrank}.
\end{proof}

\section{Jiang-Su absorption}

In this section we discuss about the heredity for the Jiang-Su absorption 
for an inclusion of unital C*-algebras with the tracial Rokhlin property.

\vskip 3mm

\begin{dfn}(\cite{HO})
A unial C*-algebra $A$ is said to be tracially $\mathcal{Z}$-absorbing if 
$A \not\cong \C$ and for any finite set $F \subset A$ and non-zero positive element $a \in A$ 
and $n \in \N$ there is an order zero construction $\phi \colon M_n \rightarrow A$ such that the following hold:
\begin{enumerate}
\item  $1 - \phi(1) \preceq a$,
\item
For any normalized element $x \in M_n$ and any $y \in F$ we have $\|[\phi(x), y)]\| < \varepsilon$.
\end{enumerate}
\end{dfn}

\vskip 3mm

\begin{thm}\label{thm:HO}(\cite[Theorem~4.1]{HO})
Let $A$ be a unital, separable, simple, nuclear C*-algebra. 
If $A$ is tracially $\mathcal{Z}$-absorbing, then $A \cong A \otimes \mathcal{Z}$.
\end{thm}

\vskip 3mm

Note that for a simple unital C*-algebra $A$ if $A$ is $\mathcal{Z}$-absorbing, 
then $A$ is tracially $\mathcal{Z}$-absorbing (\cite[Proposition~2.2]{HO}).

\vskip 3mm

\begin{thm}\label{thm:Jang-Su}
Let $P \subset A$ be an inclusion of unital C*-algebra and $E$ be a 
conditional expectation from $A$ onto $P$ with index finite type. 
Suppose that $A$ is simple, separable, unital, tracially $\mathcal{Z}$-absorbing and 
$E$ has the tracial Rohklin property. Then $P$ is tracially $\mathcal{Z}$-absorbing.
\end{thm}

\begin{proof}
Take any  finite set $F \subset P$ and non-zero positive element $a \in P$
and $n \in \N$. 
Since $E \colon A \rightarrow P$ has the tracial Rokhlin property, 
$E$ has the Rokhlin property or $A$ has the Property (SP). 
If $E$ has the Rokhlin property, then $P$ is $\mathcal{Z}$-absorbing (\cite{OT}), 
and we are done.
Hence we may assume that $A$ has the Property (SP).

Since $A$ is  simple and has the Property (SP), 
$P$ has the Property (SP) by Proposition~\ref{prp:SP-property}. 
Then there exists orthogonal projections $p_1, p_2$ in $\overline{aPa}$.

Since $A$ is tracially $\mathcal{Z}$-absorbing, there is an order zero construction $\phi \colon M_n \rightarrow A$ such that the following hold:
\begin{enumerate}
\item  $1 - \phi(1) \preceq p_1$,
\item
For any normalized element $x \in M_n$ and any $y \in F$ we have $\|[\phi(x), y)]\| < \varepsilon$.
\end{enumerate}

Since $E \colon A  \rightarrow P$ has the tracial Rokhlin, there is a projection $e \in A' \cap A^\infty$ 
satisfying $({\mathrm{Index}E})E^\infty(e) = g$ is a projection and 
$1 - g \preceq p_2$. Moreover, there is an injective *homomorphism $\beta$ 
from $A$ into $gP^\infty g$ such that $\beta(1) = g$ and $\beta(a) = ag$ for $a \in P$. 

Then a function $\beta \circ \phi (= \psi) \colon M_n \rightarrow P^\infty$ is an order zero map such that 
\begin{enumerate}
\item[(i)]
\begin{align*}
1 - \psi(1) &= 1 - (\beta \circ \phi)(1) \\
            &= 1 - g + \beta(1 - \phi(1)) \\
            &\preceq p_2 + \beta(p_1)\\
            &= p_2 + p_1\beta(1)\\
            &\preceq a 
\end{align*}, that is, $1 - \psi(1) \preceq a$ in $P^\infty$.
\item[(ii)] For any normalized element $x \in M_n$ and $y \in F$
\begin{align*}
\|[\psi(x), y]\| &= \|[\beta(\phi(x)), y]\|\\
&= \| \beta(\phi(x))y - y\beta(\phi(x))\|\\
&= \| \beta(\phi(x))\beta(y) - \beta(y)\beta(\phi(x))\|\\
&= \| \beta(\phi(x)y - y \phi(x))\|\\
&\leq \|\phi(x)y - y\phi(x)\| < \varepsilon.
\end{align*}
\end{enumerate}

Since $C^*(\phi(M_n))$ is semiprojective by \cite[Proposition~3.2 (a)]{Winter:Covering dimension I}, 
there is a $k \in \N$ and a *-homomorphism 
$\tilde{\beta}\colon C^*(\phi(M_n))) \rightarrow \Pi P/\oplus_{i=1}^kP \rightarrow P^\infty$ such that 
$\pi_k \circ \tilde{\beta} = \beta$, where $\pi_k$ be the canonical map 
from $\Pi P/\oplus_{i=1}^k P$ to $P^\infty$. 
Write $\tilde{\beta}(x) = (\tilde{\beta_n}(x)) + \oplus_{i=1}^k P$ and $g = (g_l)$ for some projection $g_l \in P$ 
for $l \in \N$.
We have, then, for sufficient large $l$ 
there is an order zero map $\tilde{\beta_l} \circ \phi \colon M_l \rightarrow P$ 
such that 

\begin{itemize}
\item[(iii)] 
\begin{align*}
1 - \tilde{\beta_l}\circ \phi(1)
&= 1 - \tilde{\beta}_l(\phi(1))\\
&= 1 - g_l + g_l - \tilde{\beta}_l(\phi(1))\\
&= 1 - g_l + \tilde{\beta}_l(1 - \phi(1)) \ (\tilde{\beta}_l(1) = g_l) \\
&\preceq 1 - g_l + \tilde{\beta}_l(p_1)\\
&\preceq p_2 + p_1',
\end{align*}
where $p_1' \in \overline{g_lPg_l}$ is projection such that $\tilde{\beta}_l(p_1) \sim p_1'$. 
Note that since $\|\tilde{\beta}_l(p_1) - p_lg_l\|$ is very small, there are projections $p_1' \in \overline{g_lPg_l}$ and 
$p_1^{''} \in \overline{p_1Pp_1}$  such that $\tilde{\beta}_l(p_1) \sim p_1' \sim p_1^{''}$. 
Since $p_2 \bot p_1^{''}$, from \cite[1.1~Proposition]{Cu} 
$p_2 + p_1' \preceq p_2 + p_1^{''}$. Hence we have 
\begin{align*}
1 - \tilde{\beta_l}\circ \phi(1) &\preceq p_2 + p_1^{''}\\
&\preceq a
\end{align*}
\item[(iv)] For any normalized element $x \in M_n$ and $y \in F$ 
$\|[\tilde{\beta_n}\circ \phi(x), y]\| < 3 \varepsilon$.
\end{itemize}
This implies that $P$ is the tracially $\mathcal{Z}$-absorbing.
\end{proof}

\vskip 3mm

The following is the main theorem in this note.

\begin{thm}\label{cor:Jiang-Su absorption}
Let $P \subset A$ be an inclusion of unital C*-algebra and $E$ be a 
conditional expectation from $A$ onto $P$ with index finite type. 
Suppose that $A$ is simple, separable, nuclear, $\mathcal{Z}$-absorbing and 
$E$ has the tracial Rohklin property. P is $\mathcal{Z}$-absorbing.
\end{thm}

\begin{proof}
It follows from Theorem~\ref{thm:Jang-Su} and Theorem~\ref{thm:HO}.
\end{proof}

\vskip 3mm

\begin{cor}
Let $A$ be an infinite dimensional simple separable unital C*-algebra and 
let $\alpha\colon G \rightarrow \Aut(A)$ be an action of a finite group $G$ 
with the tracial Rokhlin property. Suppose that $A$ is $\mathcal{Z}$-absorbing. 
Then we have 
\begin{enumerate}
\item (\cite{HO}) The fixed point algebra $A^\alpha$ and the crossed product $A \rtimes_\alpha G$
are $\mathcal{Z}$-absorbing.
\item For any subgroup $H $ of $G$ the fixed point algebra $A^H$ is $\mathcal{Z}$-absorbing.
\end{enumerate}
\end{cor}

\begin{proof}
(1): 
Since the canonical conditional expectation $E\colon A \rightarrow A^\alpha$ has the 
tracial Rokhlin property by Proposition~\ref{prp:tracial for finite group}, $A^\alpha$ is $\mathcal{Z}$-absorbing 
by Corollary~\ref{cor:Jiang-Su absorption}. 

Let $|G| = n$. Then $A \rtimes_\alpha G$ is isomorphic to $pM_n(A^\alpha)p$ for some 
projection $p \in M_n(A^\alpha)$. Since $A^\alpha$ is $\mathcal{Z}$-absorbing, $pM_n(A^\alpha)p$ 
is $\mathcal{Z}$-absorbing \cite{TW}, hence $A \rtimes_\alpha G$ is $\mathcal{Z}$-absorbing.

(2): Since $\alpha_{|H}\colon H \rightarrow \Aut(A)$ has the tracial Rokhlin property by 
\cite[Lemma~5.6]{ESWW}, we know that $A^H$ is $\mathcal{Z}$-absorbing by (1).
\end{proof}

\section{Cuntz-equivalence for inclusions of C*-algebras}

In this section we study the heredity for Cuntz equivalence for an inclusion of unital
C*-algebras with the tracial Rokhlin proerty.

Let $M_\infty(A)^+$ denote the disjoint union 
$\cup_{n=1}^\infty M_n(A)^+$. For $a \in M_n(A)^+$ and $b \in M_m(A)^+$ 
set $a \oplus b = \mathrm{diag}(a, b) \in M_{n+m}(A)^+$, 
and write $a \preceq b$ if there is a sequence $\{x_k\}$ in $M_{m,n}(A)$
such that $x_k^*bx_k \rightarrow a$. Write $a \sim b$ if 
$a \preceq b$ and $b \preceq a$. 
Put $W(A) = =M_\infty(A)^+/\sim$, and let $\langle a\rangle \in W(A)$ 
be the equivalence class containing $a$. 
Then $W(A)$ is a positive ordered abelian semigroup with equipped with the relations:

\begin{align*}
\langle a\rangle + \langle b \rangle = \langle a \oplus b\rangle, \quad
\langle a \rangle \leq \langle b\rangle \Leftrightarrow a \preceq b, 
\quad a, b \in M_\infty(A)^+.
\end{align*}

We call $W(A)$ a Cuntz semigroup.


\vskip 3mm

\begin{lem}\label{lem:Cuntz equivalent}
Let $P \subset A$ be an inclusion of unital C*-algebras with index finite type
and $E\colon A \rightarrow P$ has the tracial Rokhlin property.
Suppose that positive elements $a, b \in P^\infty$ satisfy $eb = be$ and $a \preceq eb$ in $A^\infty$, where 
$e$ is the Rokhlin projection for $E$.
Then $a \preceq b$ in $P^\infty$.
\end{lem}

\vskip 3mm

\begin{proof}
Since $a \preceq b$ in $A^\infty$, there is a sequence $\{v_n\}_{n\in \N}$ in $A^\infty$ such that 
$\|a - v_n^*ebv_n\| \rightarrow 0 \ (n \rightarrow \infty)$.

Let $E\colon A \rightarrow P$ be a conditional expectation of index finite type. 
Set $w_n = (\mathrm{Index} E)^{\frac{1}{2}}E^\infty(ev_n)$ for each $n \in \N$. 
Then, since
\begin{align*}
w_n^*bw_ne_P &= (\mathrm{Index} E)E^\infty(v_n^*e)bE^\infty(ev_n)e_P\\
&= (\mathrm{Index} E)E^\infty(v_n^*eb)E^\infty(ev_n)e_P\\
&= (\mathrm{Index} E)e_Pv_n^*ebe_Pev_ne_P\\
&= (\mathrm{Index} E)e_Pv_n^*bee_Pev_ne_P\\
&= e_Pv_n^*bv_ne_P\\
&= E^\infty(v_n^*bv_n)e_P,
\end{align*}
$w_n^*bw_n = E^\infty(v_n^*bv_n)$. Therefore,

\begin{align*}
\|a - w_n^*bw_n\| &= \|a - E^\infty(v_n^*ebv_n)\|\\
&= \|E^\infty(a - v_n^*ebv_n)\|\\
&\leq \|a - v_n^*ebv_n\| \rightarrow 0 \ (n \rightarrow \infty).
\end{align*}

This implies that $a \preceq b$ in $P^\infty$.
\end{proof}

\vskip 3mm

\vskip 3mm

\begin{prp}\label{prp:Cuntz equivalent}
Let $P \subset A$ be an inclusion of unital C*-algebras with index finite type.
Suppose that $E\colon A \rightarrow P$ has the tracial Rokhlin property. 
If two positive elements $a, b \in P$ satisfy $ a \preceq b$ in $A$, then 
$a \preceq b$ in $P$.
\end{prp}

\begin{proof}
Let $a, b \in P$ be positive elements such that $a \preceq b$ in $A$ and $\varepsilon > 0$.
Since for any constant $K > 0$ $a \preceq b$ is equivalent to $Ka \preceq Kb$, we may assume that 
$a$ and $b$ are contractive. When $b$ is invertible, then $a = (a^{1/2}b^{-1/2})b(a^{1/2}b^{-1/2})^*$, 
and $a \preceq b$ in $P$. Hence we may assume that $b$ has $0$ as a spectrum.

Since $a \preceq b$ in $A$, there is $\delta > 0$ and $r \in A$ such that 
$f_\varepsilon(a) = rf_\delta(b)r^*$, where 
$f_\varepsilon\colon \R^+ \rightarrow \R^+$ by 
$$
f_\varepsilon(t) = \left\{\begin{array}{ll}
0,&t \leq \varepsilon\\
\varepsilon^{-1}(t - \varepsilon), &\varepsilon \leq t \leq 2\varepsilon\\
1, &t \geq 2 \varepsilon.
\end{array}
\right.
$$
Set $a_0 = f_\delta(b)^{1/2}r^*rf_\delta(b)^{1/2}$. Then 
$f_\varepsilon(a) \sim a_0$. Set a continuos function $g_\delta(t)$ on $[0,1]$ by
$$
g_\delta(t) = \left\{\begin{array}{ll}
\delta^{-1}(\delta - t) &0 \leq t \leq \delta\\
0 & \delta \leq t \leq 1.
\end{array}
\right.
$$
Since $b$ has $0$ as a spectrum, $g_\delta(b) \not= 0$ and $g_\delta(b)f_\delta(b) = 0$. 
Note that $g_\delta(b) (= c)$  belongs to $\overline{bPb}$. 
Therefore, 
there are positive elements $a_0$ in $\overline{bAb}$ and $c$ in $\overline{bPb}$ 
such that $(a - \varepsilon)_+ \preceq a_0 + c$ in $A$. Indeed,
\begin{align*}
(a - \varepsilon)_+ &\preceq f_\varepsilon(a)\\
&\sim a_0\\
&\preceq a_0 + c. \ (\hbox{\cite[Proposition~1.1]{Cu}})
\end{align*}

Take a Rokhlin projection $e \in A' \cap A^\infty$ for $E$. Then there is a projection $g \in P' \cap P^\infty$ 
such that $(1 - g) \preceq ec$. Hence $(a - \varepsilon)_+(1 - g) \preceq ec$ in $A^\infty$.
By Lemma~\ref{lem:Cuntz equivalent} $(a - \varepsilon)_+(1 -g) \preceq c$ in $P^\infty$.

We have then in $P^\infty$

\begin{align*}
(a - \varepsilon)_+ &= (a - \varepsilon)_+g + (a - \varepsilon)_+(1 - g)\\
&= \beta((a - \varepsilon)_+) + (a - \varepsilon)_+(1 - g)\\
&\preceq \beta(a_0) + (a - \varepsilon)_+(1 - g) \   \ (\hbox{\cite[Proposition~1.1]{Cu}})\\
&\preceq \beta(a_0) + c \in \overline{bP^\infty b}, \ (\hbox{\cite[Proposition~1.1]{Cu}})
\end{align*}
where $\beta \colon A \rightarrow gP^\infty g$ by Lemma~\ref{lem:embedding for tracial}.
Hence $(a - \varepsilon)_+ \preceq b$ in $P^\infty$.

Since $\varepsilon > 0$, we have $a \preceq b$ in $P^\infty$, 
and $a \preceq b$ in $P$.
\end{proof}

\vskip 3mm

\section{Strict comparison property}

In this section we study the strictly comparison property for a Cuntz semigroup
and show that for an inclusion $P \subset A$ of exact, unital C*-algebras with the tracial Rokhlin property 
if $A$ has strictly comparison, then so does $P$. When $E \colon A \rightarrow P$ has the Rokhlin property, 
the statement is proved in \cite{OT2}.

\vskip 3mm

A dimension function on a C*-algebra $A$ is a function $d\colon M_\infty(A)^+ \rightarrow \R^+$
which satisfies $d(a \oplus b) = d(a) + d(b)$, and $d(a) \leq d(b)$ if $a \preceq b$ 
for all $a, b \in M_\infty(A)^+$. If $\tau$ is a positive trace on $A$, then 
$$
d_\tau(a) = \lim_{n\rightarrow\infty}\tau(a^\frac{1}{n}) = \lim_{\varepsilon \rightarrow 0+}\tau(f_\varepsilon(a)), \quad a \in M_\infty(A)^+
$$
defines a dimension function on $A$.
Every lower semicontinuous dimension function 
on an exact C*-algebra arises in this way (\cite[Theorem~II.2.2]{BH}, 
\cite{Ha}, \cite{Kir}). 
For the Cuntz semigroup $W(A)$ 
an additive order preserving mapping $\tilde{d}\colon W(A) \rightarrow \R^+$ 
is given by $\tilde{d}(\langle a\rangle) = d(a)$ from a dimension function $d$ on $A$.
We use the same symbol to denote the dimension function on $A$ and the corresponding state on $W(A)$.

Recall that an C*-algebra $A$ has strict comparison if whenenver $x, y \in W(A)$ are 
such that $d(x) < d(y)$ for all dimension function $d$ on $A$, then $x \leq y$.
If $A$ is simple, exact and unital , then   the strictly comparison property 
is equivalent to the strictly comparison proeprty by traces, that is, for all $x, y \in W(A)$ 
one has that $x \leq y$ if $d_\tau(x) < d_\tau(y)$ for  all tracial states $\tau$ 
on $A$ (\cite[Corollary~4.6]{Rordam:Z-absorbing}).

Let $\mathrm{T}(A)$ be the set of all traces on a C*-algebra $A$.

\vskip 3mm

\begin{prp}\label{prp:tracial states}
Let $E\colon A \rightarrow P$ be of index finite type and has the tracial Rokhlin property.
Then the restriction map defines a bijection from the set $\mathrm{T}(A)$ to the set $\mathrm{T}(P)$. 
\end{prp}

\begin{proof}
Since $E\colon A \rightarrow P$ has the tracial Rokhlin property, 
for any $n$ there are a projection $g_n \in P^\infty$ and an injective *-homomorphism $\beta_n$
from $A$ to $g_nP^\infty g_n$ such that $\beta(x) = xg_n$ for all $x \in P$ and 
$\tau^\infty(1 - g_n) < \frac{1}{n}$ for any $\tau \in \mathrm{T(P)}$. 
Then we show that the map 
$\mathrm{T(P)} \ni \tau \mapsto \mathrm{w^*-}\lim(\tau^\infty \circ \beta_n) \in \mathrm{T(A)}$ 
is an inverse of the restriction map, 
where $\tau^\infty$ be the extended tracial state of $\tau$ on $P^\infty$ and 
$\mathrm{w^*}-$ means the weak*-limit. 

Let $R$ be the restriction map of $\tau \in \mathrm{T(A)}$ to $P$. 
For $\tau \in T(P)$ $\tau^\infty \circ \beta_n$ is a tracial state on $A$. 
Then for any $a \in P$
\begin{align*}
R(\mathrm{w^*-}\lim(\tau^\infty \circ \beta_n))(a) 
&= \mathrm{w^*-}\lim(\tau^\infty(ag_n))\\
&=\tau^\infty(a)\\
&= \tau(a).\\
\end{align*}
Hence  $R$ is surjective.

Suppose that $R(\tau_1) = R(\tau_2)$ for $\tau_1, \tau_2 \in T(A)$.
Since $E\colon A \rightarrow P$ is of index finite type, there exists a quasi-basis 
$\{(u_i, u_i^*)\}_{i=1}^n \subset A \times A$ such that 
for any $x \in A$ 
$$
x = \sum_{i=1}^nE(xu_i)u_i^* = \sum_{i=1}^nu_iE(u_i^*x).
$$

Let $\tau_1 - \tau_2 = (\tau_1 - \tau_2)_{+} - (\tau_1 - \tau_2)_{-}$ be the Jordan decomposition of 
$\tau_1 - \tau_2$. 
Since $E$ is completely positive, $\tau \circ E$ is completely positive. 
We have then by the Schwarz inequality
\begin{align*}
|(\tau_1-\tau_2)_{+}(E(xu_i)u_i^*)\tau(E(xu_i)u_i^*)^*| &\leq |(\tau_1-\tau_2)_{+}(E(xu_i)u_i^*(E(xu_i)^*u_i^*)^*)|\\
&= |(\tau_1 - \tau_2)_{+}(E(xu_i)u_i^*u_iE(xu_i)^*)|\\
&\leq \|u_i\|^2|(\tau_1 - \tau_2)_{+}(E(xu_i)E(xu_i)^*)| = 0.
\end{align*}
Hence $(\tau_1 - \tau_2)_+(E(xu_i)u_i^*) = 0$. 
This means that 
\begin{align*}
(\tau_1 - \tau_2)_{+}(x) &= \sum_{i=1}^n(\tau_1 - \tau_2)_{+}(E(xu_i)u_i^*)  = 0. 
\end{align*}
Hence $(\tau_1 - \tau_2)_{+} = 0$. 
Similarly, $(\tau_1 - \tau_2)_{-} = 0$. Hence, $\tau_1 - \tau_2 = 0$, and $R$ is injective.
\end{proof}

\vskip 3mm

\vskip 3mm

\begin{thm}\label{thm:ComparisonTheorem}
Let $P \subset A$ be an inclusion of unital C*-algebras with index finite type.
Suppose that $A$ is simple and exact and $A$ 
has strict comparison, and $E\colon A \rightarrow P$ 
has the tracial Rokhlin property. 
Then $P$ has strict comparison.
\end{thm}

\begin{proof}
Since $E\colon A \rightarrow P$ is of index finite type and $A$ is simple, 
$P$ is exact and simple. Note that the strict comparison property is equivalent to the 
the strict comparison property given by traces, i.e., 
for all $x, y \in \mathcal{C}(A)$ one has that $x \leq y$ if $d_\tau(x) < d_\tau(y)$ for all tracial states $\tau$ 
on $A$ (see \cite[Remark~6.2]{RW} \cite{Rordam:Z-absorbing}).

Since  $E \otimes id\colon A \otimes M_n \rightarrow P \otimes M_n$ is of index 
finite type and has the tracial Rokhlin property, it suffices to verify the condition that  
whenever $a, b \in P$ are positive elements such that $d_\tau(a) < d_\tau(b)$ for 
all $\tau \in \mathrm{T(P)}$, then $a \preceq b$. 

Let $a, b \in P$ be positive elements such that $d_\tau(a) < d_\tau(b)$ for 
all $\tau \in \mathrm{T(P)}$. Then by Proposition~\ref{prp:tracial states}
we have $d_\tau(a) < d_\tau(b)$ for all tracial states $\tau \in \mathrm{T}(A)$. 
Since $A$ has strict comparison, $x \preceq y$ in $A$. Therefore
By Proposition~\ref{prp:Cuntz equivalent} $a \preceq b$ in $P$, and $P$ has strict comparison.
\end{proof}

\vskip 3mm

\begin{cor}\label{thm:ComparisonTheorem}
Let $P \subset A$ be an inclusion of unital C*-algebras with index finite type.
Suppose that $A$ is simple and the order on 
projections on $A$ is determined by traces, and $E\colon A \rightarrow P$ 
has the tracial Rokhlin property. 
Then the order on projections on $P$ is determined by traces.
\end{cor}

\begin{proof}

Since  $E \otimes id\colon A \otimes M_n \rightarrow P \otimes M_n$ is of index 
finite type and has the tracial Rokhlin property, it suffices to verify the condition that  
whenever $p, q \in P$ are projections such that $\tau(p) < \tau(q)$ for 
all $\tau \in \mathrm{T(P)}$, then $p$ is the Murray-von Neumann equivalent to a subprojection of $q$ in $P$.

Let $p, q \in P$ be projections such that $\tau(p) < \tau(q)$ for 
all $\tau \in \mathrm{T(P)}$. Then by Proposition~\ref{prp:tracial states}
we have $\tau(p) < \tau(q)$ for all tracial states $\tau \in \mathrm{T}(A)$. 
Since the order on projections on $A$, $p$ is the Murray-von Neumann subequivalnt to $q$ in $A$, and 
$p \preceq q$ in $A$ by \cite[Proposition~2.1]{Ro:UHF2}. Therefore
By Proposition~\ref{prp:Cuntz equivalent} $p \preceq q$ in $P$, 
and $p$ is the Murray-von Neumann equivalent to a subprojection of $q$ in $P$.
Hence the order on projections on $P$ is determined by traces.
\end{proof}

\vskip 3mm

The following is well-known.

\vskip 3mm

\begin{lem}\label{lem:comparison}
let $A$ be an exact C*-algebra and $p$ be a projection of $A$.
Suppose that $A$ has strictly comparison. Then so does $pAp$.  
\end{lem}

\begin{proof}
Since for each $n \in \N$  $M_n(A)$ is exact and has strictly comparison, we have only to show that 
whenever $a, b \in pAp$ are positive elements such that $d_\tau(a) < d_\tau(b)$ for all tracial states $\tau \in \mathrm{T}(pAp)$,
then $a \preceq b$. 

Let $a, b \in pAp$ be positive elements such that $d_\tau(a) < d_\tau(b)$ for all tracial states $\tau \in \mathrm{T}(pAp)$. 
Then for any tracial $\rho \in \mathrm{T}(A)$ the restriction $\rho_{|pAp}$ belongs to $\mathrm{T}(pAp)$. 
Since $d_\rho(a) = d_{\rho_{|pAp}}(a) < d_{\rho_{|pAp}}(b) = d_\rho(b)$, we have $a \preceq b$ in $A$ by the assumption.
Hence there is a sequence $\{x_n\} \subset A$ such that $\|x_n^*bx_n - a\| \rightarrow 0$. 
Set $y_n = px_np \in pAp$ for each $n \in \N$, then 
\begin{align*}
\|y_n^*by_n - a \| &= \|px_n^*pbpx_np - pap\|\\
&= \|px_n^*bx_np - pap\|\\
&\leq \|x_n^*bx_n - a\| \rightarrow 0,
\end{align*}
and $a \preceq b$ in $pAp$. Therefore, $pAp$ has strict comparison.
\end{proof}

\vskip 3mm

\begin{cor}
Let $A$ be an infinite dimensional simple separable unital C*-algebra and 
let $\alpha\colon G \rightarrow \Aut(A)$ be an action of a finite group $G$ 
with the tracial Rokhlin property. Suppose that $A$ is exact and has strict comparsion. 
Then we have 
\begin{enumerate}
\item \cite{HO} The fixed point algebra $A^\alpha$ and the crossed product $A \rtimes_\alpha G$
have strict comparioson.
\item For any subgroup $H $ of $G$ the fixed point algebra $A^H$ has strict comparison.
\end{enumerate}
\end{cor}

\begin{proof}
(1) Since the canonical conditional expectation $E\colon A \rightarrow A^\alpha$ has the tracial Rokhlin property 
by Proposition~\ref{prp:tracial for finite group}, $A^\alpha$ has strict comparison 
by Theorem~\ref{thm:ComparisonTheorem}. 

Let $|G| = n$. Then $A \rtimes_\alpha G$ is isomorphic to $pM_n(A^\alpha)p$ for some 
projection $p \in M_n(A^\alpha)$. Since $A^\alpha$ has strict comparison, $pM_n(A^\alpha)p$ 
has strict comparison by Lemma~\ref{lem:comparison}, hence $A \rtimes_\alpha G$ has strict comparison.

(2): Since $\alpha_{|H}\colon H \rightarrow \Aut(A)$ has the tracial Rokhlin property by 
\cite[Lemma~5.6]{ESWW}, we know that $A^H$ has strict comparison by (1).
\end{proof}

\vskip 3mm

Similarly we have the following.

\vskip 3mm

\begin{cor}
Let $A$ be an infinite dimensional simple separable unital C*-algebra and 
let $\alpha\colon G \rightarrow \Aut(A)$ be an action of a finite group $G$ 
with the tracial Rokhlin property. Suppose that the order on projections on $A$ is 
determined by traces. 
Then we have 
\begin{enumerate}
\item \cite{HO} The order on projections on the fixed point algebra $A^\alpha$ and the crossed product $A \rtimes_\alpha G$
is determined by traces.
\item For any subgroup $H $ of $G$ the order on projections on the fixed point algebra $A^H$ is determined by traces.
\end{enumerate}
\end{cor}

{\bf Acknowledgement}
The authors thank Professor Yasuhiko Sato for fruitful discussion, in particlular,
about Lemma~\ref{lem:MNequivalent}.


\end{document}